\RequirePackage{fix-cm}
\documentclass{elsarticle}
\usepackage{graphicx}

\newcommand{\Z}{\mathbb{Z}}
\newcommand{\F}{\mathbb{F}}
\newcommand{\supp}{{\rm supp}}
\newcommand{\Hom}{{\rm Hom}}

\newcommand{\dd}{{\bf d}}

\newcommand{\xy}{\gamma}
\newcommand{\f}{{\cal F}}
\newcommand{\cc}{{\cal C}}
\newcommand{\A}{{\cal A}}
\usepackage{mathptmx}      
\usepackage{latexsym}
\usepackage{amsfonts,amsmath}
\usepackage{tikz}
\usetikzlibrary{fit,calc}
\usetikzlibrary{backgrounds}
\newtheorem{theorem}{Theorem}
\newtheorem{remark}{Remark}
\newtheorem{lemma}[theorem]{Lemma}
\newtheorem{corollary}[theorem]{Corollary}
\newdefinition{definition}{Definition}
\newdefinition{example}{Example}
\newproof{proof}{Proof}

\begin{document}
\begin{frontmatter}
\title{Induced Weights on Quotient Modules and an Application to Error Correction in Coherent Networks\thanks{}
}


\author{Eimear Byrne}
\ead{ebyrne@ucd.ie} 

\address{School of Mathematics and Statistics,\\
           University College Dublin, Ireland}



\begin{abstract}
We consider distance functions on a quotient module $M/K$ induced by distance functions on a module $M$.
We define error-correction for codes in $M/K$ with respect to induced distance functions.
For the case that the metric is induced by a homogeneous weight, we derive analogues of the Plotkin and Elias-Bassalygo bounds and give their asymptotic versions. These results have applications to coherent network error-correction in the presence of adversarial errors.
We outline this connection, extending the linear network coding scheme introduced by Yang {\em et al}.

\end{abstract}

\begin{keyword}network code \sep network error-correction \sep coherent networks \sep Plotkin bound \sep Elias bound \sep finite Frobenius ring \sep homogeneous weight
\end{keyword}
\end{frontmatter}

\section{Introduction}

Coding in data communication networks has been shown to offer many advantages in terms of data rate, error correction and security. 
Many coding models have been considered for a variety of networks. 
Error-correction in {\em coherent} network coding has been considered in \cite{KM03,RK17,YNY07,YY07,YYN11,Z08}. For non-coherent networks, where the network topology is unknown, subspace codes have been shown to offer good solutions for error correction and have been widely studied. This is also often referred to as {\em random} network coding, albeit a different notion of random network coding as introduced by Ho {\em et al} in \cite{ho}.  
  
We consider the set-up for coherent network coding described in \cite{YNY07,YY07,YYN11}. In coherent network coding the network topology is known and the {\em encoding vectors} are often chosen deterministically, according to connectivity of the network. In \cite{jaggi}, the authors describe a polynomial-time deterministic algorithm to generate a linear code for error-free networks, which is extended in \cite{YYN11} for networks with errors. 

The network is described as a directed acyclic graph with a single source node and several sink nodes, or receivers. The source transmits some $m$ data packets, one for each edge with which it is incident, all $m$ of which are to be delivered to each sink. This is referred to as {\em multicast}. Successful delivery of all data packets to a sink requires at least $m$ edge disjoint paths from the source to the sink.
In a linear network coding scheme, at each node in the network linear combinations of packets on its incoming edges are transmitted along its outgoing edges. Transfer of data from source to sink nodes may be described by a {\em transfer matrix}, which can be assumed to be invertible if its coefficients are chosen from a large enough ring or field. Full details of this approach for a finite field alphabet may be read in \cite{KM03}. A transfer matrix $F$ is constructed as $F=(I-K)^{-1}$, where an entry of $K$ is non-zero only if the corresponding entry of the adjacency matrix of the line graph of the network is non-zero, that is, if $K$ `fits' this adjacency matrix. Each column of $F$ corresponds to an edge in the network; a particular sink only `sees' the columns of $F$ corresponding to the edges incident with it. The projection of $F$ onto these columns yields a matrix whose row space is a linear code and the sink can retrieve all $m$ packets if the particular $m$-subset of its rows (corresponding to the source packets) are linearly independent.

Error correction in the coherent network coding model is an important and challenging problem. A single (Hamming) error introduced in a link can propagate through the network infecting many other links. However, inherent redundancy in the network means that some error patterns are invisible to a given sink node, specifically, those error patterns corresponding to elements orthogonal to its code. Aside from these irrelevant errors, the code of a sink node may have some error correction capability.

In an effort to quantify these properties of a linear network code, in \cite{YYN11} the authors present refinements of the sphere-packing, Singleton and Gilbert-Varshamov bounds for an arbitrary linear error-correcting network code over a finite field.
The error-model considered is for {\em adversarial errors}, that is, where it is assumed that an adversary has access to some number of edges in the network. 
The bounds in \cite{YYN11} and the references therein are implicitly bounds on the size of codes in a quotient vector space,
with respect to a distance function induced by the Hamming distance on the original vector space. 

In this work we continue this line of research, establishing refinements of the Plotkin and Elias-Bassalygo bounds for codes in quotient modules. Our results are of independent interest and can be expressed without direct reference to a network. However, in the context of network coding, we extend the framework of \cite{YYN11} to present a model that holds not only for finite field alphabets and Hamming errors, but also for their natural coding theoretic generalizations. When reduced to the finite field case, the bounds given here outperform the sphere-packing and Singleton bounds for some parameters, similar to the comparison of the Plotkin and Elias-Bassalygo bounds for classical Hamming codes.

In our setting, the underlying alphabet is a finite bimodule and the codes are equipped with a distance function induced by the {\em homogeneous weight}.
This very general model includes the case where the alphabet is a finite field and the induced metric comes from the Hamming weight and so is a strict extension of the coding scheme described in \cite{YYN11}. More importantly, we develop upper bounds on codes associated with a fixed network based on the Plotkin and Elias bounds. These results are new both for the classical case of a finite field and in the more general case of a finite Frobenius bimodule\footnote{The results given here were presented at the International Workshop in Coding and Cryptography, Bergen in 2013 \cite{B13}. The finite field case has been considered independently in \cite{WJ15}.}.

We will close this section with a motivating example from network coding, to show how codes in quotient spaces may arise in applications. In Section 2
we give preliminary facts on rings and modules relevant to the paper, and introduce weight functions on quotient modules induced from other weight functions. In Section 3 we discuss the homogeneous weight and its behaviour on Frobenius bimodules. In particular we note the character theoretic description of such weights in this case, which we require for the results given in Section 4. 
The main results of the paper are in Section 4, where we give analogues of the Plotkin and Elias-Bassalygo bounds for codes in quotient modules with respect to the induced homogeneous weight. We furthermore give asymptotic versions of these bounds. In Section 5 we describe how codes in quotient modules fit into a model of coherent network error-correction.

\subsection{A Motivating Example from Network Coding}

We give a simple example, in advance of giving a formal description of a network code.    

\begin{example}\label{ex:toy}
	Consider the following network, whose source $s$ emits $2$ packets $x_1$ and $x_2$ along the edges $a_1$ and $a_2$, and which has two sink nodes $t_1$ and $t_2$ as shown. The aim is that each sink should be able to retrieve both transmitted packets from the data it collects from its incident edges.	
	
	\vspace{3mm}
	\includegraphics[scale=0.75]{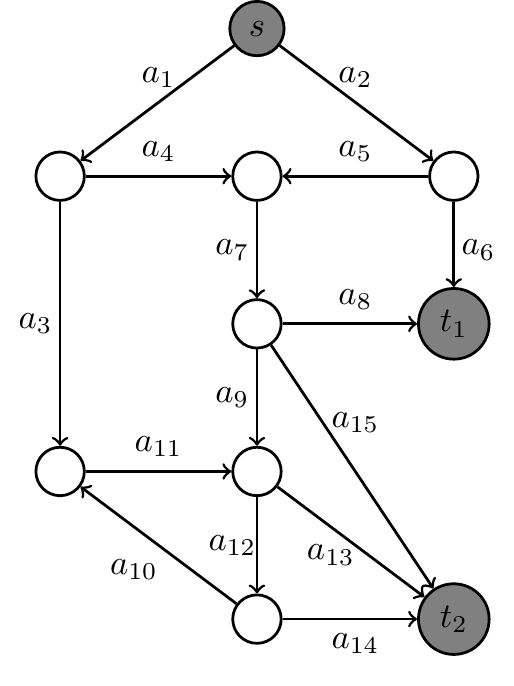}

	
	The line-graph of the network has adjacency matrix as shown below on the left. If we treat this as a matrix $K$ over $\F_2$ then one possible invertible transfer matrix for the network is given by $F=(I-K)^{-1}$, as shown below on the right.
	
	$$K=\left(\begin{array}{c}
	001100000000000\\
	000011000000000\\
	000000000010000\\
	000000100000000\\
	000000100000000\\
	000000000000000\\
	000000011000001\\
	000000000000000\\
	000000000001100\\
	000000000010000\\
	000000000001100\\
	000000000000010\\
	000000000000000\\
	000000000000000\\
	000000000000000
	\end{array}
	\right),\:\: F=(I-K)^{-1}
	=\left(\begin{array}{c}
	1     0     1     1     0     0     1     1     1     0     1     0     0     0     1\\
	0     1     0     0     1     1     1     1     1     0     0     1     1     1     1\\
	0     0     1     0     0     0     0     0     0     0     1     1     1     1     0\\
	0     0     0     1     0     0     1     1     1     0     0     1     1     1     1\\
	0     0     0     0     1     0     1     1     1     0     0     1     1     1     1\\
	0     0     0     0     0     1     0     0     0     0     0     0     0     0     0\\
	0     0     0     0     0     0     1     1     1     0     0     1     1     1     1\\
	0     0     0     0     0     0     0     1     0     0     0     0     0     0     0\\
	0     0     0     0     0     0     0     0     1     0     0     1     1     1     0\\
	0     0     0     0     0     0     0     0     0     1     1     1     1     1     0\\
	0     0     0     0     0     0     0     0     0     0     1     1     1     1     0\\
	0     0     0     0     0     0     0     0     0     0     0     1     0     1     0\\
	0     0     0     0     0     0     0     0     0     0     0     0     1     0     0\\
	0     0     0     0     0     0     0     0     0     0     0     0     0     1     0\\
	0     0     0     0     0     0     0     0     0     0     0     0     0     0     1
	\end{array}
	\right)$$
	
	The word $x=(x_1,x_2,0,0,0,0,0,0,0,0,0,0,0,0,0) \in \F_2^{15}$ corresponds to the source transmitting packets $x_i$ along its incident edges $a_i$, $i=1,2$. This word traverses the network as $$c=xF= (x_1,x_2,x_1,x_1,x_2,x_2,x_1+x_2,x_1+x_2,x_1+x_2,0,x_1,x_2,x_2,x_2,x_1+x_2).$$
	The sink $t_1$ received the projection onto the 6th and 8th columns, which is the subvector $(x_2,x_1+x_2)$, while the receiver at $t_2$ gets the projection onto coordinates $13,14,15$ and so receives $(x_2,x_2,x_1+x_2)$. Both receivers can recover the source packets $x_1$ and $x_2$. 
	We may associate with each sink $t_i$ a code $C_i$. If we let $\F_2^2$ be the message space for each sink then $C_1= \F_2^2$ is an $\F_2$-$[2,2]$ code with generator matrix 
	$G_1 =\left[\begin{array}{c} 01\\ 11  \end{array}\right] $, found by projecting onto the 1st 2 rows and columns 6 and 8 of $F$. $C_2$ is an $\F_2$-$[3,2]$ code with generator matrix 
	$G_2 =\left[\begin{array}{c} 001\\ 111  \end{array}\right]$  found by projecting onto the 1st 2 rows and last 3 columns of $F$. In fact, we embed $\F_2^2$ into $\F_2^{15}$
	to get a message space ${\cal M} = \{(x_1,x_2,0,0,...,0): x_i \in \F_2\}$. We call $F_1$ the $15 \times 2$ matrix consisting of the 6th and 8th columns of $F$ and we call $F_2$ the $15 \times 3$ matrix consisting of the last 3 columns of $F$. 
	Then $C_i = \{xF_i: x \in {\cal M} \}$.
	
	Now suppose that an adversary can corrupt any links in the network by injecting errors. This may be represented by an error vector $e \in \F_2^{15}$, where $e_i$ is non-zero if an error occurs on the arc $a_i$ of the network. The error $e$ propagates through the network as $eF$ and the codeword $c=xF$ is corrupted, resulting in the network word 
	$(x+e)F$. The sink at $t_1$ gets the received word
	$ (x+e)F_1
	=(x_2+e_2+e_6,x_1+x_2+e_1+e_2+e_4+e_5+e_7+e_8)$. If the error vector $e$ has the form
	$e = (e_1,e_2,e_3,e_4,e_5,e_2,e_7,e_1+e_2+e_4+e_5+e_7,e_9,e_{10},...,e_{15}) $,
	that is, if $e$ is in the left nullspace $K_1$ of $F_1$, then the received word is $(x_2,x_1+x_2)$, which can be decoded. Otherwise, sink $t_1$ will incorrectly decode its received packets as $x_1+e_1+e_4+e_5+e_6+e_7+e_8$ and $x_1+x_2+e_1+e_2+e_4+e_5+e_7+e_8$.    
	Errors in the kernel $K_1$ are `invisible' to $t_1$. Errors that lie outside $K_1$ cannot be corrected by the receiver. Therefore, it makes sense to identify the code $C_1$ as a subspace of the quotient space $\F_2^{15} / K_1$ and to measure the distance of an error to a codeword in this quotient space. For a distance function on $\F_2^{15}$, the induced weight of a coset $w+K_1$ may be defined to be the minimum weight over all elements of $w+K_1$ with respect to this original distance function (the weight of a coset leader).
	With respect to the Hamming metric, every non-zero word of $C_1$ has induced weight $1$, since every coset of $K_1$ contains a word of Hamming weight 1 and hence $C_1$ corrects no errors.
	The code $C_2$ at $t_2$ can tolerate erasures in the first and second coordinates but not in the third.
	Let $K_2=\{ e \in \F_2^{15}: eF_2=0 \}$, which has dimension 13 in $\F_2^{15}$.
	The 4 codewords of $C_2$ correspond to the following cosets of $K_2$ in $\F_2^{15}$:
	
	\begin{center}
	\begin{tabular}{r|r|c}
		\hline
		$C_2$ &  $\F_2^{15} / K_2$ & Weight \\
		\hline
		000 & $K_2$ & 0\\
		001 & 10 00000 00000 000 + $K_2$& 1\\
		110 & 00 10000 00000 000 + $K_2$& 1\\
		111 & 01 00000 00000 000 + $K_2$& 1 \\
		\hline
	\end{tabular}
    \end{center}
	
	$C_2$ is itself a subspace of $\F_2^{3}$ and again all its elements have weight 1 with respect to the weight induced by the Hamming metric, so $C_2$ is 0-error-correcting. However, the receiver at $t_2$ can detect certain error patterns, namely those $v$ for which $vF_2$ has different coordinates in the first two positions.   
\end{example}

\section{Modules, Quotients and Induced Metrics}

In terms of the network code, which we define formally in Section \ref{sec:code}, each constituent code will have the bimodule $\cal A$ as its alphabet and the intermediate node operations will be $R$-linear. For the weight function we consider in later sections, we will assume further properties of ${\cal A}$, namely that it be a {\em Frobenius} bimodule.

\subsection{Bimodules over a Finite Ring}

Let $R$ be a finite ring with unity and let ${\cal A}$ be a finite $R$-$R$ bimodule. 
This means that ${\cal A}$ is both a left and right $R$-module satisfying $(ra)r=r(ar)$
for all $r \in R$ and $a \in {\cal A}$. We write $_R{\cal A}_R$ to indicate that ${\cal A}$ is an $R$-$R$ bimodule. The reader is referred to \cite{AW92,kasch} for further background material on rings and modules.

\begin{example}
	
	We give some examples of $R$-$R$ bimodules ${\cal A}$. There are many.
	\begin{enumerate}
		\item $R = {\cal A}=\F_q$
		
		\item $R=\F_q$, ${\cal A} = \F_q^{k \times \ell}$
		
		\item $R=\F_q$, ${\cal A} = \F_{q^\ell}$
		
		\item $R={\cal A}$ (Any ring $R$ is a bimodule over itself.)
		
		\item If $R$ is a subring of a finite ring $S$ then $S={\cal A}$ is an $R$-$R$ bimodule. 
		
		\item $R={\cal A} = \F_q$, $\Z_4$, $\Z_k$, $\F_q^{k \times \ell}$, $GR(p^\ell,r)$.
		
		\item If $R$ is commutative then any left $R$-module is an $R$-$R$ bimodule.
		
		\item The tensor product $U \otimes V$ over $\Z$ is an $R$-$R$ bimodule for a left $R$-module $U$ and right $R$-module $V$.
		
		\item For any finite ring $R$, its character module $\hat{R}:={\rm Hom}(R,{\mathbb C}^\times)$ (the homomorphisms from $(R,+)$ to the multiplicative group of ${\mathbb C}$) is an $R$-$R$ bimodule.
	\end{enumerate}

\end{example}

\subsection{Submodules, Homomorphisms and Induced Weights}

\begin{definition}
	Let $R$ be a finite ring and let $M$ and $N$ be left $R$-modules. 
	An $R$-homomorphism $f$ from $_RM$ to $_RN$ (i.e. as left $R$-modules) is a map $f:M\longrightarrow N$ such that 
	$$f(r_1a+r_2b)=r_1f(a)+r_2f(b)$$ for any $r_1,r_2 \in R, a \in M$.
	An $R$-homomorphism is defined similarly for right $R$-modules $M$ and $N$. 
	If $M$ and $N$ are both $R$-$R$ bimodules then
	an $R$-homomorphism $f$ from $_RM_R$ to $_RN_R$ (i.e. as bimodules) is a map $f:M\longrightarrow N$ such that 
	$$f(r_1as_1+r_2bs_2)=r_1f(a)s_1+r_2f(b)s_2$$ for any $r_1,r_2,s_1,s_2 \in R, a \in M$. 
	In other words, $f$ is an $R$-homomorphism from $M$ to $N$ as both left and right $R$-modules.
\end{definition}

We'll write $\Hom(_RM,_RN)$ and $\Hom(M,N)$ to denote the $R$-linear left and bimodule homomorphisms, respectively from $M$ to $N$. 
A submodule $K$ of a bimodule $M$ is both a left and right $R$-submodule of $M$.

We define a weight function, or weight on an $R$-module $M$ to be a map ${\bf w} : M \longrightarrow {\mathbb R}$ such that ${\bf w}(0)=0$.

\begin{definition}
	Let $M$
	be a left (resp. right, resp. bi-) modules over a finite ring $R$. 
	Let $K$ be a left (resp. right, resp. bi-) submodule of $M$. 
	Let ${\bf w}$ be a weight function  on $M$. The weight function on the quotient module $M/K$ induced by ${\bf w}$ 
	is defined to be
	\begin{eqnarray*}
		\hat{\bf w}(x)&: =& \min \{{\bf w}(z) : z \in x + K \}
	\end{eqnarray*}    
\end{definition}

A weight function ${\bf w}$ on an $R$-module $M$ determines a map $\dd$ on $M \times M$ via $\dd(x,y):={\bf w}(x-y)$. 
It can be checked that if $\dd$ is a distance function on $M$ then the corresponding induced function $\hat{\dd}$ on Im$f$ defined by 
$\hat{\dd}(u,v):=\hat{\bf w}(u-v)$ is also a distance function.

\begin{definition}
	Let $\delta >0$. Let $M$ be a left (resp. right, resp. bi-) module over a finite ring $R$ and let $K$ be an $R$-submodule of $M$. 
	Let $\dd$ be a distance function on $M$ and let $\hat{\dd}$ be the distance function on $M/K$ induced by $\dd$.
	We say that a non-empty subset $C$ of $M/K$ is a $\delta$-error correcting code with respect to $\hat{\dd}$ 
	if for any $z +K\in M/K$ and $c+K \in C$ satisfying $\hat{\dd}(z+K,c+K)<\delta$, it holds that
	$$\hat{\dd}(z+K,c+K)< \hat{\dd}(z+K,c'+K)$$ for all $c'+K \in C$.
\end{definition}

  Now suppose that we have $R$-modules $K,M$ with $K$ a left (resp. right, resp. bi-) submodule of $M$.
  Let $N$ be a left (resp. right, resp. bi-) $R$-module and let $f$ be an $R$-module epimorphism 
  from $M$ to $N$ as left (resp. right, resp. bi-) modules, with kernel $K$. 
  Then a distance function $\dd$ on $M$ gives rise to an induced distance function $\bar{\dd}$ on $N$ defined by
  $\bar{\dd}(f(x),f(y)) = \hat{\dd}(x+K,y+K)$. In particular, we have an isometry between 
  $(M/K,\hat{\dd})$ and $(N,\bar{\dd})$.
  
   Let $n$ be a positive integer. We extend a weight function ${\bf w}$ of an $R$-module ${\cal A}$ to a weight function on ${\cal A}^n$ in the obvious way:
  $${\bf w}: {\cal A}^n \longrightarrow {\mathbb R}:{\bf w}(c) \mapsto \sum_{i=1}^n {\bf w}(c_i),$$
  for any $c=(c_1,\ldots,c_n) \in {\cal A}^n$.
  Obviously, not all weight functions on ${\cal A}^n$ can be derived in this way, but for the bounds we derive later we restrict to such weight functions. 
  With respect to the application to coherent network coding we describe later,
  we will take $M ={\cal A}^n$, $N={\cal A}^N$ for a bimodule ${\cal A}$.                   	

\section{Modules and Homogeneous Weights}

The homogeneous weight was first introduced on the ring ${\mathbb Z}_m$ in \cite{CH97}. It is a very a natural generalization of the Hamming weight. On $\Z_4$ it is given by the Lee weight which yields an isometry to $\F_2^2$ under the Hamming weight. Generalizations of this weight function have appeared in \cite{GNW04,GS00,HN99}. In coding theory, codes whose alphabet is a finite Frobenius ring, or more generally is a finite Frobenius bimodule, play an important role. This is because these are the largest class of rings or modules such that fundamental theorems such as those extending the MacWilliams duality theorem and the MacWilliams extension theorem hold for ring-linear codes (see \cite{GHMcFWZ,GNW04,W99,W09} and associated references for further details).


\subsection{Homogeneous Weight Functions}
Recall that a weight on an $R$-module $M$ is a map ${\bf w} : M \longrightarrow {\mathbb R}$ such that ${\bf w}(0)=0$. 
The homogeneity conditions in \cite{GNW04} are given by the following.

\begin{definition}
    A weight function ${\bf w}$ on a left $R$-module $M$ is called (left) homogeneous if
     \begin{enumerate}
  \item[H1] If $Rx=Ry$ then ${\bf w}(x)= {\bf w}(y)$ for all $x,y\in M$.
  \item[H2] 
 There exists a real number $\gamma$ 
  such that
    \begin{eqnarray}
  \sum_{y\in Rx}{\bf w}(y)  = \gamma \, |Rx| \:\:\forall \:\:0 \neq x \in M .
\end{eqnarray}
\end{enumerate}  
\end{definition}
Right homogeneous weights are defined similarly.
The value $\gamma$ is the average weight of a cyclic left $R$-submodule. The homogeneity property says that this average value is independent 
A homogeneous weight function exists on any finite $R$-module and is unique up to choice of average weight $\gamma$ \cite{HN99}.

\begin{example}
Let $R=M=\F_q$. The Hamming weight is homogeneous on $M$ with $\gamma = \frac{q-1}{q}$.
\end{example}

\begin{example}
	Let $R=M={\mathbb Z}_4$. The Lee weight is homogeneous on $M$ with $\gamma = 1$. 
\end{example}

\begin{example}
	Let $R=M={\mathbb Z}_{p^k}$ for a prime $p$. The weight function defined by
	$$ {\bf w}(x) = \left\{ \begin{array}{cl}
	0             & \text{ if } x = 0 \\
	\frac{p}{p-1} & \text{ if } x \in p^{k-1} \Z_{p^k},\\
	1             & \text{ if otherwise,}
	\end{array} 
	\right.
	$$
	is homogeneous on $={\mathbb Z}_{p^k}$. 
\end{example}	

\begin{example}
   Let $R=M = \F_q^{2 \times 2}$. The weight function ${\bf w}$ on $M$ defined by
   $$ {\bf w}(x) = \left\{ \begin{array}{cl}
                       \frac{q^2 - q - 1}{q-1} & \text{ if } {\rm{rank}}_{\F_q}(x) = 2, \\
                       q                    & \text{ if } {\rm{rank}}_{\F_q}(x) = 1,\\
                       0                                      & \text{ if } x=0,
                   \end{array} 
             \right.
   $$
   is homogeneous with average value $\gamma = \frac{q^2-1}{q}.$  
   For the case $q=2$, this is the Bachoc weight \cite{bachoc}.
\end{example} 
  \begin{example} 
   For the case $R=\F_q, \;M=\F_q^{2 \times 2}$, the Hamming weight is homogeneous for $\gamma = \frac{q-1}{q}$.               
\end{example}

In \cite{GNW04}, the authors show that every finite unital ring $R$ has a quasi-Frobenius bimodule, which is unique up to right and left $R$-isomorphism if the sum of its minimal left submodules is cyclic both as a left and right $R$-module. Such a module is then called a Frobenius bimodule.

\subsection{Homogeneous Weights and Characters of a Bimodule}

With respect to the $R$-$R$ bimodule ${\cal A}$ we write $\hat{{\cal A}}$ to denote Hom$_{\mathbb Z}({\cal A},{\mathbb C}^{\times})$, the group of characters of the additive group of ${\cal A}$. Then $\hat{{\cal A}}$
is an $R$-$R$ bimodule according to the relations
$$^r{\chi}(x) = \chi(xr),\:\:\: \chi^r(x) = \chi(rx)$$ 
for all $r \in R, x \in {\cal A}$ and $\chi \in \hat{\cal A}$. 
A character $\chi \in \hat{{\cal A}}$ is called (left) generating
if given any $\phi \in \hat{R}$ there is some $r \in R$ satisfying 
$\phi = {^r}{\chi}$. This is equivalent to the property that $\ker \chi$ contains no non-zero
left $R$-submodule of ${\cal A}$. In \cite{W09} the authors show that every finite ring possesses a {\em quasi-Frobenius} bimodule and that this is moreover unique up to module isomorphism if its {\em socle} is cyclic as a module over $R$, in which case it is called a {\em Frobenius} bimodule. We refer the interested reader to that paper and the references given there for a detailed exposition. For the purposes of this paper, it is enough to observe the behaviour of the homogenous weight on such bimodules.  
    
         \begin{definition}
         The bimodule $_R{\cal A}_R$ is called a Frobenius bimodule if  
         $${_R} {\cal A} \cong {_R}\hat{R} \:\:\text{ and }\:\: {{\cal A}_R}  \cong \hat{R}_R. $$ 
         \end{definition}
     
This gives a construction of a Frobenius bimodule $_R{\cal A}_R $ for a finite ring $R$: choose ${\cal A}$ to be the character module $\hat{R}$ of $R$. Any other Frobenius $R$-$R$ bimodule is isomorphic to $\cal A =\hat{R}$. If $R$ is a Frobenius ring then ${\cal A}=R$ itself is the canonical choice.  
 
By duality, if $_R{\cal A}_R$ is Frobenius then ${_R} \hat{\cal A} \cong {_R}{R} \:\:\text{ and }\:\: {\hat{\cal A}_R}  \cong {R}_R. $ 
In particular, if $_R{\cal A}_R$ is a Frobenius bimodule then $\hat{\cal A}$ is generated by a character $\chi$ both as a left and as a right $R$-module. It is exactly this property that makes the homogeneous weight behave so nicely on a Frobenius ring or bimodule, and is the reason why we restrict to this (large) class of modules. 
	More precisely, the existence of such a generating character $\chi \in \hat{\cal A}$ gives the following characterisation of the homogeneous weight on a Frobenius bimodule (c.f. \cite{H00}).

\begin{lemma}\label{lemcharfrob}
  Let $_R{\cal A}_R$ be a Frobenius bimodule with generating character $\chi$. Then the weight function
  $${\bf w}: {\cal A} \longrightarrow {\mathbb R} : a \mapsto \gamma\left(1-\frac{1}{|R^{\times}|}\sum_{u\in R^{\times}}
  \chi(au)\right)$$ is homogeneous. 
\end{lemma} 

\subsection{A Key Lemma}

For a positive integer $n$, word $z \in {\cal A}^n$ and set $X \subset \{1,...,n\}$ we define the projection of $z$ onto the coordinates not in $X$ by
$$\pi_{X}(z):=(z_i)_{i \notin X} \in {\cal A}^{n-|X|}.$$ 
In other words, $\pi_X(x)$ is obtained by puncturing $x$ on the coordinates of $X$. 

Given an $R$-submodule $M < {\cal A}^n$, we define the support of $M$ to be the set 
$$\supp(M):=\{ i : z_i \neq 0 \text{ for some } z\in M  \}.$$

Using the character-theoretic description of the homogeneous weight given above we following result. The proof proceeds almost exactly as in \cite[Lemma 1]{bgks10}.

\begin{lemma}\label{lemavwsp}
   Let $_R{\cal A}_R$ be a Frobenius bimodule with homogenous weight function ${\bf w}:{\cal A} \longrightarrow {\mathbb R}$.
   Let $n$ be a positive integer, let $M$ be an $R$-submodule of ${\cal A}^n$ and let $x \in {\cal A}^n$. Then
   $$\sum_{c \in M} {\bf w}(x+c)  = \xy |M| |\supp(M)|+|M| {\bf w}(\pi_{\supp(M)}(x)).$$
\end{lemma}
In particular, if the module $M$ has full support then the average homogeneous weight of any coset of $(M,+)$ has the same constant value $\xy  n$.

Unless stated otherwise, for the remainder we will assume that ${\cal A}$ is a Frobenius $R$-$R$ bimodule and that ${\bf w}$ denotes the homogeneous weight on ${\cal A}^k$ for any positive integer $k$.

\section{Upper Bounds on Codes in Quotient Modules}\label{sec:main}

For the remainder of Section \ref{sec:main} we fix the following notation. Let $n$ be a positive integer, let 
and let $K$ be an submodule of of the $R$-$R$ bimodule ${\cal A}^n$. 
Let $M$ be a union of cosets of $K$ in ${\cal A}^n$ and let 
$$C:=\{ u + K : u \in M\} \subset {\cal A}^n / K.$$
Let $\hat{\bf w}$ be the weight function on ${\cal A}^n / K$ induced by ${\bf w}$ with respect to $K$,
so $\hat{\bf w}(x) = {\bf w}(u+K) := \min\{ {\bf w}(v) : v \in u+K \}$. 
Let $d$ be the minimum distance of $C$ with respect to the weight function $\hat{\bf w}$ that is,
$$ d = \min \{ \hat{\bf w}(x-y) : x,y \in C, x\neq y \} = \min \{ {\bf w}(u-v) : u,v \in M, u-v \notin K\}.$$
We furthermore fix 
$$s:=|\supp(M)| \text{ and }\ell:=|\supp(K)|.$$
Then $|{\cal A}|^{s}\geq |M|,|{\cal A}|^{\ell}\geq |K| .$

Now $\supp(K) \subset \supp(M)$, so clearly, $s \geq \ell$.  
If we puncture $M$ and $K$ on the coordinates of $M$ not in $\supp(M)$,
then the resulting punctured subsets $M'$ and $K'$ of ${\cal A}^s$ satisfy $|M|=|M'|$ and $|K|=|K'|$.
Let $C'=\{u+K':u \in M'\}\subset {\cal A}^s / K' $. 
Since $|M|=|C||K|$ and $|M'|=|C'||K'|$, we have $|C|=|C'|$. For this reason, we assume throughout the next two subsections that $M$ has full support and so $s=n$, since otherwise we can simply derive upper bound on $|C'|$, which has full support. In the application to network coding, we cannot assume that $n=s$, but for the purposes of the bounds we derive here, there is no loss of generality in making this assumption.

We now present two new upper bounds on $|C|$.
Our results are generalizations of the Plotkin bound and of the Elias-Bassalygo bound for linear block codes and invoke these bounds for codes constructed as quotient modules. As in the classical case, the Plotkin-like bound is applicable for codes with high minimum distance ($d > \gamma s$) and the Elias-Bassalygo bound is applicable for code with minimum distance upper-bounded by $\gamma s$.
Our results hold in particular for codes defined over finite fields (and are new in this case) with respect to the metric induced by the Hamming weight, but of course holds for a  much larger class of alphabets. 

\subsection{A Plotkin-Like Upper Bound}

Following the usual argument for the classical Plotkin bound, we obtain lower and upper
bounds on $\displaystyle{\sum_{c,c' \in C} \hat{\dd}(c,c')}$.
Proposition 2.1 of \cite{GS04} gives a Plotkin bound for codes over finite Frobenius rings with respect to the homogeneous weight, which has the following easy extension for a code in ${\cal A}^s$.

\begin{lemma}\label{lemgsplotkin}
	Let ${\cal C} \subset {\cal A}^s$ have support size $s$ and have minimum homogeneous distance $\dd({\cal C})$. Then
	$$ |{\cal C}|(|{\cal C}|-1) \dd({\cal C}) \leq \sum_{x,y \in {\cal C}} {\bf w}(x-y) \leq \xy s |{\cal C}|^2.$$
\end{lemma}

We are now ready to prove our first main result, which relies on the Lemma \ref{lemgsplotkin} (the Plotkin bound for the homogeneous weight) and on Lemma \ref{lemavwsp}, which computes the average homogeneous weight of a coset.

\begin{theorem}[Plotkin Bound]\label{thplotkin}
	\begin{eqnarray}\label{eq:plot}
	(|C| - 1) d  \leq \left( |C| s   - \ell \right) \xy.
	\end{eqnarray}
	If $d > \xy s$ then $$\displaystyle{|C| \leq \frac{d - \xy \ell}{d-\xy s}}.$$
\end{theorem}

\begin{proof}
	We give an estimate of the sum of the distances between ordered pairs of distinct codewords of $C$.
	\begin{eqnarray*}
		|C|(|C| - 1) d  & \leq & \sum_{x,y \in C} \hat{\bf w}(x-y), \\
		& = &  \sum_{u+K,v+K \in C} {\bf w}(u-v+K), \\
		& \leq & \sum_{u+K,v+K \in C} \frac{1}{|K|} \sum_{k \in K} {\bf w}(u-v+k), \\
		& = & \frac{1}{|K|^2} \sum_{ \begin{tiny}\begin{array}{c}u,v \in M: \\ u-v \notin K \end{array}\end{tiny}} {\bf w}(u-v),
	\end{eqnarray*}
	\begin{eqnarray*}
		& = & \frac{1}{|K|^2} \left( \sum_{u,v \in M} {\bf w}(u-v) -  \sum_{\begin{tiny}\begin{array}{c}u,v \in M: \\ u-v \in K \end{array}\end{tiny}} {\bf w}(u-v) \right), \\
		& = & \frac{1}{|K|^2} \left( \sum_{u,v \in M} {\bf w}(u-v) - |C| |K| \sum_{z \in K} {\bf w}(z) \right), \end{eqnarray*}
	We compute the left-hand sum using Lemma  \ref{lemgsplotkin} and the right-hand sum using Lemma \ref{lemavwsp}, which yields: 
	\begin{eqnarray*}
		|C|(|C| - 1) d & \leq & |C|^2 s \xy   -|C| \ell \xy,
	\end{eqnarray*}
	from which we get (\ref{eq:plot}).
	In the case that $d > \xy s$ we rearrange (\ref{eq:plot}) to obtain the following upper bound on the size of $C$:
	$$|C| \leq \frac{d - \xy \ell}{d-\xy s}.$$
\end{proof}

\begin{example}
	Let $R = {\cal A} = \Z_4$. The Lee weight $\dd_L$ is homogeneous on $\Z_4$ with $\xy=1$.
	Let $K = \langle 0111333 \rangle \subset \Z_4^7$, which has support size $6$. Then the distance function $\hat{\bf d}_L$ on $\Z_4^7 / K$ induced by $\dd_L$
	is the minimum Lee weight of each coset of $K$ in $\Z_4^7$. 
	If $M \subset \Z_4^7$ has support size 7 and $C = \{x + K : x \in M\} \subset \Z_4^7 / \langle 0111333 \rangle$, 
	then the Plotkin bound tells us that
	$|C| \leq (\hat{\bf d}_L(C) - 6)/(\hat{\bf d}_L(C) - 7)$. Then the maximum size of any such code $C$ with $\hat{\bf d}(C)=8$ is 2.
	Let 
	$$M=\{1 0 2 2 0 1 2, 1 2 0 0 2 3 0,1 3 1 1 1 2 3,1 1 3 3 3 0 1 \} \cup \{3 3 3 1 3 2 1,3 1 1 3 1 0 3,3 0 0 2 2 1 0,3 2 2 0 0 3 2 \},$$
	which is a union of 2 cosets of $K$.
	Then $$C=\{1022012+ \langle 0111333 \rangle , 3331321 + \langle 0111333 \rangle\},$$ 
	has cardinality 2 and minimum distance $8$, since 
	the minimum Lee distance between elements of the different cosets is $8$.
	Therefore, $C$ is optimal with respect to Theorem \ref{thplotkin}. 
	
	For a concrete realization of $C$ as a submodule of $\Z_4^6$, let $f$ be a $\Z_4$-homomorphism of $\Z_4^7$ that has
	$K$ as its kernel. 
	For example, say $f$ is the map with following matrix representation in $\Z_4^{7 \times 6}$.
	$$\left(\begin{array}{c}
	1 0 0 0 0 0\\
	0 1 0 0 0 0\\
	0 0 1 0 0 0\\
	0 0 0 1 0 0\\
	0 0 0 0 1 0\\
	0 0 0 0 0 1\\
	0 1 1 1 3 3\\
	\end{array}\right) $$  
	Then the elements of $C$ are in 1-1 correspondence with the elements of 
	${\bar C}=f(M) =  \{120023,300221 \} \subset \Z_4^6$.
	The Lee weight on $\Z_4^7$ along with the $\Z_4$-submodule $K$ induces 
	an isometry between $(\Z_4^7/K,\hat{\dd}_L)$ and $(\Z_4^6,\bar{\dd}_L)$
	whereby $\bar{\dd}_L(f(x)) = \hat{\dd}_L(x)$.
\end{example}

\begin{remark}
	If $\ell=s$ then the inequality (\ref{eq:plot}) implies that $d \leq s \xy$.
	On the other hand, if $K$ is trivial then $\ell = 0$ and so Theorem \ref{thplotkin} is the Plotkin bound for the homogeneous weight \cite[Theorem 2.2]{GS04}. 
\end{remark}

\subsection{An Elias Bassalygo Bound for Quotient Modules}\label{sec:elias}

We now give an upper bound on $|C|$ for the case $d < \gamma s$.
We recall the following well-known lemma (see, for example \cite[Lemma 5.2.9]{vL99}).
\begin{lemma}\label{lem:AB}
	Let $M$ be an additive group and let $A, B$ be subsets of $M$. Then there exists $x \in M$ such that
	$$\displaystyle{|B|\leq\frac{|(x + A) \cap B||M|}{|A|} }.$$ 
\end{lemma}

For each nonnegative real number $r$ we define
$$\hat{B}^{{\rm av}} (r):= \{ z \in {\cal A}^{s}/K : \frac{1}{|K|}\sum_{x \in z+K} {\bf w}(x) \leq r \},$$
which is the set of words $z$ such that the average homomgeneous weight in the coset $z+K$ is at most $r$.

For a positive integer $k$, we denote by $B^k(r) \subset {\cal A}^k$ the standard sphere of radius $r>0$ about zero with respect to the homogeneous weight $\bf w$ on ${\cal A}$, that is $$B^k(r) := \{ z \in {\cal A}^k :{\bf w}(z) \leq r \}.$$ Clearly for $R=\F_q={\cal A}$, this is the usual Hamming sphere.
Setting $A$ to be a translate of $\hat{B}^{{\rm av}} (r)$ and $B=C$ we deduce from Lemma \ref{lem:AB} that 
\begin{equation}\label{eqe2}
| C| \leq 
\frac{|\hat{B}^{{\rm av}}  (r) \cap C| |{\cal A}|^s}{|\hat{B}^{{\rm av}}  (r)||K|} .
\end{equation}
For each 
vector $x \in {\cal A}^{s}$, 
we define 
$$ \pi(x):=  (x_i)_{i \notin \supp K} \in {\cal A}^{s-\ell} \text{ and }
\pi'(x):= (x_i)_{i \in \supp K} \in {\cal A}^{\ell}.$$
Observe that if $x+K=y+K$ for some $x,y \in {\cal A}^s$ then $\pi(x)=\pi(y)$.

\begin{corollary}\label{coravball}
	Let $r \geq \ell \xy$. Then
	$$\hat{B}^{{\rm av}} (r)= \{ z + K\in {\cal A}^{s}/K :  {\bf w}(\pi(z)) \leq r - \ell \xy \}.$$
\end{corollary}

\begin{proof}
	Since any $x,y \in {\cal A}^s$ in the same coset of $K$ satisfy $\pi(x)=\pi(y)$, we see that the set on the right-hand-side is well defined.
	Let $z+K \in {\cal A}^{s}/K$. Applying Lemma \ref{lemavwsp}, $z +K\in\hat{B}^{{\rm av}} (r)$ if and only if 
	$$r\geq \frac{1}{|K|}\sum_{x \in  z+K} {\bf w}(x)=\xy \ell+{\bf w}(\pi(x)),$$
	which is true if and only if ${\bf w}(\pi(z)) \leq r-\xy \ell$.
\end{proof}

\begin{corollary}\label{cor:balls}
	Let $r \geq \xy\ell $. Then
	$$|\hat{B}^{{\rm av}} (r)|=|B^{s-\ell}(r-\xy \ell)||\A|^{\ell}/|K|$$
\end{corollary}

\begin{proof}
	Let $x \in \A^s$. 
	From Corollary \ref{coravball}, $x+K$ is contained in $\hat{B}^{{\rm av}} (r)$ if and only if $\pi(x) \in B^{s-\ell}(r-\xy \ell)$.
	Let $K':=\{ \pi'(z): z \in K \} \subset \A^{\ell}$. Clearly $|K'|=|K|$,
	so there are exactly $|\A|^{\ell}/|K|$ distinct cosets of $K'$ in $\A^{\ell}$. Then 
	for each $u \in B^{s-\ell}(r-\xy \ell)$, there are exactly $|\A|^{\ell}/|K|$ distinct elements
	$z + K \in \A^{s}/K$ satisfying $\pi(z)=u$. 
\end{proof}

Applying (\ref{eqe2}), along with Corollary \ref{cor:balls} we now have the following result.

\begin{theorem}\label{th:el1} Let $ r - \xy \ell > 0 $. Then
	\begin{equation}\label{eqe3}
	|C| \leq \frac{| \hat{B}^{\rm av}(r) \cap C||\A|^{s-\ell}  }{|B^{s-\ell}(r-\xy \ell)| }.
	\end{equation}
\end{theorem}

We now obtain an upper bound on the size of $\hat{B}^{\rm av} (r) \cap C$.
First we define the anti-code
$$C(r):= \hat{B}^{\rm av} (r) \cap C.$$ 
Clearly, $C(r)$, being a subset of $C$, has minimum induced distance $\hat{\dd}(C(r)) \geq d$.

We denote by $C^{\pi}(r)$ the multiset of words in $\A^{s-\ell}$ constructed by applying the puncturing map $\pi$ to exactly one word from each distinct coset $x+K$ in $C(r)$, that is
$$C^{\pi}(r):=\{ \pi(x_1),...,\pi(x_{|C(r)|}) : x_i-x_j \notin K \text{ if } i \neq j\},$$
and the sum of the multiplicities of the members of $C^\pi(r)$ is $|C(r)|$. 
We will compute a lower bound on the sum of the distances between pairs of elements of $C^\pi(r)$, counting multiplicities.

\begin{theorem}\label{th:lbw} 
	Let $r\geq\xy \ell$. Then
	$$ |C(r)|(|C(r)|-1)(d-\xy \ell) \leq \sum_{u,v \in C^{\pi}(r)}  {\bf w}(u-v).   $$
\end{theorem}

\begin{proof}
	Using Lemma \ref{lemavwsp}, we obtain
	\begin{eqnarray*}
		|C(r)|(|C(r)|-1)d & \leq &  \sum_{x+K,y+K \in C(r)} \hat{\bf w}(x-y+K) \\
		& \leq &  |C(r)|(|C(r)|-1)\xy \ell + \sum_{x+K,y+K \in C(r)}  {\bf w}(\pi(x-y)) ,\\
		&   =  &  |C(r)|(|C(r)|-1)\xy \ell +  \sum_{u,v \in C^{\pi}(r)}  {\bf w}(u-v).  
	\end{eqnarray*}
\end{proof}


\begin{lemma} Let $r\geq\xy \ell$. Then 
	$\dd(C^{\pi}(r)):=\min \{\dd(x,y): x,y \in C^{\pi}(r), x\neq y  \} \geq d - \xy \ell$.
\end{lemma}

\begin{proof}
	As before, let $K':=\{ \pi'(z): z \in K \} \subset \A^{\ell}$, which is simply the projection of $K$ onto its own support.
	Let $x+K,y+K \in C(r)$ satisfy $\pi(x) \neq \pi(y)$.
	Then
	\begin{eqnarray*}
		d \leq \hat{\dd}(x+K,y+K)  &  =  & \hat{\bf w}(x-y+K) \\
		&  =  & \hat{\bf w}(\pi'(x-y)+K') +{\bf w}(\pi(x)-\pi(y)) \\
		&\leq & \xy \ell + {\bf w}(\pi(x)-\pi(y)).
	\end{eqnarray*}
	In particular $\dd(C^\pi(r)) \geq d - \xy \ell.$
\end{proof}

We now obtain an upper bound on the sum of the weights of differences of members of $C^\pi(r)$, adapting the argument used in the classical Elias-Bassalygo bound.

\begin{theorem}\label{th:cr}
	$$\sum_{u,v \in C^\pi(r)} {\bf w}(u-v)= |C(r)|^2(r-\xy \ell)\left( 2 - \frac{r-\xy \ell}{\xy(s-\ell)} \right)$$
\end{theorem}

\begin{proof}
	We may assume that $\supp(K) = \{s-\ell+1,...,s\}$.
For each $i \in \{1,...,s\}$, let $m_i(u)$ be the multiplicity of the element $u \in {\cal A}$ 
    in the $i$th coordinate of $C^\pi(r)$. 
    Therefore,
    \begin{eqnarray*}
    \sum_{u,v \in C^\pi(r)} {\bf w}(u-v)& =& \sum_{i=1}^{s-\ell} \sum_{a,b \in {\cal A}}w(a-b)m_i(a)m_i(b) .\\
    \end{eqnarray*}
    Since ${\cal A}$ is a Frobenius bimodule the homogeneous weight takes the form given in Lemma \ref{lemcharfrob}, therefore, for each $i\in \{1,...,s-\ell\}$ we have
    \begin{eqnarray*}
    \sum_{a,b \in {\cal A} }{\bf w}(a-b)m_i(a)m_i(b)& =& \xy \sum_{a,b \in {\cal A}} 
    \left( 1 - \frac{1}{|R^\times|} \sum_{\theta \in R^\times }\chi(( a-b)\theta) \right)m_i(a)m_i(b) \\
    & = & \xy |C(r)|^2 - \xy\sum_{\theta \in R^\times } \sum_{a,b \in {\cal A}} \frac{1}{|R^\times|} \chi(( a-b)\theta))m_i(a)m_i(b) \\
    & = & \xy |C(r)|^2 -\frac{ \xy}{|R^\times|}\sum_{\theta \in R^\times }\left| \sum_{a \in {\cal A}} \chi( a\theta )m_i(a)\right|^2
    \end{eqnarray*}
    From the Cauchy-Schwarz inequality we have
    \begin{eqnarray*}
    	\sum_{\theta \in R^\times }\left| \sum_{a \in {\cal A}} \chi( a\theta )m_i(a)\right|^2& \geq & \frac{1}{|R^\times|}\left|\sum_{\theta \in R^\times } \sum_{a \in {\cal A}} \chi( a\theta )m_i(a)\right|^2 
    	 =  |R^\times| \left| \sum_{a \in {\cal A}}m_i(a)\frac{1}{|R^\times|}\sum_{\theta \in R^\times }  \chi( a\theta )\right|^2,\\
    	& = & |R^\times| \left( \sum_{a \in {\cal A}}m_i(a)\left(1-\frac{{\bf w}(a)}{\xy}\right)\right)^2
    	=  |R^\times| \left( |C(r)| -\frac{1}{\xy}\sum_{a \in {\cal A}}m_i(a){\bf w}(a)\right)^2.
    \end{eqnarray*}
    We therefore arrive at the inequality
    \begin{eqnarray*}
    	\sum_{a,b \in {\cal A} }{\bf w}(a-b)m_i(a)m_i(b)& \leq & \xy \left( |C(r)|^2 - \left( |C(r)| -\frac{1}{\xy}\sum_{a \in {\cal A}}m_i(a){\bf w}(a)\right)^2  \right),\\
    	& = & 2|C(r)| \sum_{a \in {\cal A}}m_i(a){\bf w}(a)- \frac{1}{\xy} \left( \sum_{a \in {\cal A}}m_i(a){\bf w}(a) \right)^2.
    \end{eqnarray*}
Again, using the Cauchy-Schwarz inequality we have that
\begin{eqnarray*}
    \sum_{i=1}^{s-\ell} \left( \sum_{a \in {\cal A}}m_i(a){\bf w}(a) \right)^2 & \geq & 
    \frac{1}{s-\ell} \left( \sum_{i=1}^{s-\ell} \sum_{a \in {\cal A}}m_i(a){\bf w}(a) \right)^2,
\end{eqnarray*}
from which we deduce that
\begin{eqnarray*}
\sum_{u,v \in C^\pi(r)} {\bf w}(u-v)& \leq & 2|C(r)|\sum_{i=1}^{s-\ell}  \sum_{a \in {\cal A}}m_i(a){\bf w}(a) - \frac{1}{\xy(s-\ell)} \left( \sum_{i=1}^{s-\ell} \sum_{a \in {\cal A}}m_i(a){\bf w}(a) \right)^2,\\
& = & 2|C(r)|\sum_{u\in C^\pi(r)} {\bf w}(u) - \frac{1}{\xy(s-\ell)}\left(\sum_{u\in C^\pi(r)} {\bf w}(u)\right)^2.
\end{eqnarray*}
Let $T:= \sum_{u\in C^\pi(r)} {\bf w}(u) .$ 
Recall that $z \in C(r)$ only if the average homogeneous weight of the elements of $z+K$ is at most $r$, in which case we have ${\bf w}(\pi(z)) \leq r-\gamma \ell$. It follows that $T \leq |C(r)| (r-\gamma \ell)$.
Since $2|C(r)|T - \frac{T^2}{\xy(s-\ell)}$ is increasing on $[0,|C(r)|\xy(s-\ell)]$, if $r \leq \xy s$ then $|C(r)|(r-\xy \ell)$ falls within this range, so we may write
\begin{eqnarray*}
	\sum_{u,v \in C^\pi(r)} {\bf w}(u-v)& \leq & 2|C(r)|^2(r-\xy \ell) -\frac{1}{\xy(s-\ell)}|C(r)|^2(r-\xy \ell)^2\\
	      & = & |C(r)|^2(r-\xy \ell)\left( 2 - \frac{r-\xy \ell}{\xy(s-\ell)} \right)
\end{eqnarray*}      
\end{proof}

\begin{corollary}\label{cor:elias}
	If
	$r \leq \xy s $ and $(r-\xy\ell)^2 -\xy(s-\ell)(2r-d-\xy\ell) > 0$ then
	$$|C(r)| \leq h(r,s,\ell,d):=\frac{ (d-\xy \ell)\xy (s-\ell)}{(r-\xy\ell)^2 -\xy(s-\ell)(2r-d-\xy\ell)}. $$
\end{corollary}

\begin{proof}
	Combine Theorems \ref{th:lbw} and \ref{th:cr} to get
	$$ (|C(r)|-1)(d-\xy \ell) \leq |C(r)|(r-\xy \ell)\left( 2 - \frac{r-\xy \ell}{\xy(s-\ell)} \right),$$
	which gives
	$$|C(r)| ((r-\xy\ell)^2 -\xy(s-\ell)(2r-d-\xy\ell)) \leq (d-\xy \ell)\xy (s-\ell).$$
	Therefore, under the hypothesis of the theorem, we have
	\begin{equation*}
	   |C(r)|  \leq  \frac{ (d-\xy \ell)\xy (s-\ell)}{(r-\xy\ell)^2 -\xy(s-\ell)(2r-d-\xy\ell)}.
	\end{equation*}
\end{proof}

Finally, combining Theorems \ref{th:el1} and \ref{cor:elias} we obtain the Elias-Bassalygo bound for codes in quotient modules.

\begin{theorem}[Elias-Bassalygo Bound]\label{thelias}
	Let $r \leq \xy s $ and $(r-\xy\ell)^2 -\xy(s-\ell)(2r-d-\xy\ell) > 0$.
	Then
	$$|C| \leq    \frac{h(r,s,\ell,d) |\A|^{s-\ell} }{| B^{s-\ell}(r-\xy \ell)| } .$$

\end{theorem}

\subsection{Code Optimality}\label{sec:opt}

We use $A({\cal A}:n,s,\ell,d)$ to denote the maximum size of any code $C$ in ${\cal A}^n/K$ with induced minimum homogeneous distance $d= \hat{\dd}(C)$ over all possible choices of submodules $K$ of ${\cal A}^n$ with support size $\ell$ and such that the support of
$M=\cup_{x \in C} (x+K)$ has size $s\leq n$. 

\begin{definition}
	We define ${A}({\cal A}:n,s,\ell,d):=$
	\begin{equation*}
	\max\{ |C| : {_RK}_R < {\cal A}, C \subset {\cal A}^n/K, \supp(K)=\ell,\supp(\cup_{x \in C} (x+K))=s \}
	\end{equation*}
\end{definition}

We define one more ball of positive radius $r$, again identifying $K$ with its punctured module in ${\cal A}^s$:
$$\hat{B}(r):=\{ x+K \in {\cal A}^s/K : \hat{\bf w}(z+K) \leq r  \}.$$ 
Clearly, if $2r \leq d-1$, then from the usual sphere-packing argument we have:
$$|C| \leq \frac{|{\cal A}|^s}{|K||\hat{B}(r)|}\leq \frac{|{\cal A}|^s}{|K||\hat{B^{\rm av}}(r)|},$$
since $\hat{B^{\rm av}}(\delta) \subset \hat{B}(r)$.
If $\delta \geq \xy \ell$, then using Corollary \ref{cor:balls} we get
$$|C| \leq \frac{|{\cal A}|^{s-\ell}}{|{B^{s-\ell}}(r-\xy \ell)|}.$$

We summarize the preceding as follows.

\begin{theorem}
	Let $n,s,\ell,d$ be nonnegative integers satisfying $n \geq s \geq \ell$. Then
	$$
	{A}({\cal A}:n,s,\ell,d) \leq 
	\left\{ \begin{array}{cll}
	      \displaystyle{\frac{d - \xy \ell}{d-\xy s} }& \text{ if } & d > \xy s\\
	     \displaystyle{\frac{h(r,s,\ell,d) |\A|^{s-\ell} }{| B^{s-\ell}(r-\xy \ell)| }} 
	     & \text{ if }
	     &\begin{array}{l}
	        \xy s \geq r,d > \xy \ell,\\ 
	        (r-\xy\ell)^2 -\xy(s-\ell)(2r-d-\xy\ell) > 0\\
	     \end{array}\\
	     \displaystyle{\frac{|{\cal A}|^{s-\ell}}{|{B^{s-\ell}}(r-\xy \ell)|}} & \text{ if } & \frac{d-1}{2} \geq r > \xy \ell
	\end{array}\right.
	$$
\end{theorem}

\subsection{Asymptotic Bounds}

Asymptotic versions of these bounds are expressed by finding upper bounds on:
$$\alpha(\A: \sigma,\lambda,\delta):= 
\lim_{n \rightarrow \infty}  \sup n^{-1} \log_{|\A|}\left({ A}({\cal A}:n,\sigma n,\lambda n,\delta n))\right).$$ 
As in the classical Hamming case, we will require an asymptotic expression for the size of the homogeneous sphere $B^k(\delta k)\subset \A^k$. This was essentially answered first in \cite{L94} and in a slightly different form (which we use here) in \cite[Theorem 4.1]{GS04} as follows:

\begin{lemma}
	For all $\delta \in [0,\gamma]$ there holds:
	$$ \lim_{k \rightarrow \infty}  \sup k^{-1} \log_{|\A|} |B^k(\delta k)| = \min \left\{ \log_{|\A|}  \left(\sum_{a \in \A} Z^{{\bf w}(a)-\delta}\right) : Z \in (0,1] \right\}.$$
\end{lemma}

\begin{definition}
	Let $\delta \geq 0$. We define the function
	$$ H_{\A}(\delta):=\min \left\{ \log_{|\A|}  \left(\sum_{a \in \A} Z^{{\bf w}(a)-\delta}\right) : Z \in (0,1] \right\}.$$ 
\end{definition}


\begin{theorem}[Asymptotic Plotkin Bound]\label{th:asympplotkin}
	Let $\sigma,\lambda,\delta $ satisfy $0<\delta,0\leq \lambda < \sigma \leq 1$.
	$$
	\alpha(\A: \sigma,\lambda,\delta) \leq \left\{
	\begin{array}{cll}
	0 & \rm{ if } & \delta > \gamma \sigma \\  
	\sigma -\frac{\delta}{\xy} & \rm{ if } & \delta \leq \gamma \sigma
	\end{array}
	\right.
	$$
\end{theorem}

\begin{proof}
	If $\delta > \gamma \sigma$, then from Theorem \ref{thplotkin} we have that 
	$	\alpha(\A: \sigma,\lambda,\delta)=0$.
	
	Now suppose that $\delta \leq \gamma \sigma$. 
	Let $C \subset {\cal A}^n/K$ be an optimal code and let $d=\delta n$.
	Choose $s'$ to be the greatest integer satisfying $\gamma s'\leq d-1$. Then $s-s' > 0$
	and $s' > \ell$ by our choice of $s'$.     
	Consider the words of $M \subset \A^n$. By a standard coding theory argument we can take $s-s'$ 
	successive (affine) shortenings of $M$ on its coordinates in 
	$\supp(M)\backslash \supp(K)$ to arrive at a code in $\A^n$ of order at least 
	$|M|/|\A|^{s-s'}$. Then 
	puncture the code on these coordinates, as well as any coordinate not in $\supp (M)$, 
	to obtain a code $M' \subset \A^{s'}$ of the same order. The corresponding 
	code $K'$ obtained by puncturing $K$ on the same coordinates satisfies 
	$|K'| = |K|$ and so the set of words of $M'$ is a union 
	of at least $|C|/|\A|^{s-s'}$ distinct cosets of $K'$ in $\A^{s'}$. This yields a code $C'$ satisfying 
	$d(C')=d'\geq d > \gamma s'$. Again apply Theorem \ref{thplotkin} to get 
	$$\frac{|C|}{|\A|^{s-s'}} \leq |C'| \leq \frac{d - \gamma \ell}{d-\gamma s'} \leq d - \gamma \ell 
	= (\delta  -\gamma \lambda) n.$$
	Then 
	\begin{eqnarray*}
		\alpha(\A: \sigma,\lambda,\delta) & \leq &  
		\lim_{n \rightarrow \infty}  \sup n^{-1} \log_{|\A|}|\A|^{s-s'}(\delta - \gamma \lambda) n \\
		& \leq & \lim_{n \rightarrow \infty}  \sup \frac{s-s'}{n} + \frac{\log_{|\A|}(\delta - \gamma \lambda)n}{n}\\ 
		& \leq & \sigma -\frac{\delta}{\xy} + \lim_{n \rightarrow \infty} + \frac{\log_{|\A|}(\delta - \gamma \lambda)}{n}
		+\frac{\log_{|\A|} n}{n} =  \sigma -\frac{\delta}{\xy} 
	\end{eqnarray*}      
\end{proof}


\begin{theorem}[Asymptotic Elias Bound]\label{th:asymelias} 
	Let $\rho ,\delta>0$ and let $\sigma,\lambda \in (0,1)$. Suppose that 
	$\xy \lambda \leq \rho,\delta \leq \xy \sigma$ and further that
	$ \rho < \gamma \sigma - \sqrt{\xy (\sigma - \lambda)(\xy \sigma - \delta)} $.
	Let
	$\displaystyle{\xi: = \frac{\rho -\xy \lambda}{\sigma-\lambda}}$. Then
	\begin{eqnarray*}
		\lim_{n \rightarrow \infty}  \sup n^{-1} \log_{|\A|}\left({ A}({\cal A}:n,\sigma n,\lambda n,\delta n))\right) & \leq & \sigma - \lambda - H_{\A}(\xi).
	\end{eqnarray*}
	In particular,
	$$ \alpha(\A: \sigma,\lambda,\delta)  \leq \sigma - \lambda - H_\A\left( \xy- \sqrt{\frac{ \xy(\xy \sigma - \delta)}{\sigma-\lambda}} \right).$$
\end{theorem}

\begin{proof}
	It is easy to see that $ n^{-1} \log_{|\A|}h(\rho,\sigma,u,v)  \rightarrow 0$ as $n\rightarrow 0$.
	From Theorem \ref{thelias}, for $\sigma,\lambda,\delta,\rho$ as above, we have
	\begin{eqnarray*}
	\alpha(\A: \sigma,\lambda,\delta)  & \leq & 
		\lim_{n \rightarrow \infty}  n^{-1} \log_{|\A|} \left(\frac{h(\rho,\sigma,\lambda,\delta)|\A|^{n(1-\lambda)} }{| B^{n(1-\lambda)}(n(\rho -\xy \lambda)| }  \right)\\
		& = & \sigma - \lambda - \lim_{n \rightarrow \infty} n^{-1} \log_{|\A|}(|B^{n(\sigma-\lambda)}(\xi(n(\sigma-\lambda)))|)\\
		& = & (\sigma-\lambda) - H_{\A}(\xi)
	\end{eqnarray*}
	Since this holds true for any choice of $\rho$, by continuity of $H_{\A}$, the final inequality holds.
\end{proof}

Similarly, we have an asymptotic sphere-packing bound.

\begin{theorem}[Asymptotic Sphere Packing Bound]
	Let $\sigma,\lambda,\delta $ satisfy $0<\delta,0\leq \lambda < \sigma \leq 1$ and suppose that $\delta>2\gamma \lambda$. 
	Then $$\alpha(\A: \sigma,\lambda,\delta)  \leq \sigma -\lambda - H_\A\left( \frac{\delta -2 \xy \lambda }{2(\sigma-\lambda)}\right).$$ 
\end{theorem}

\begin{remark}
	As $\lambda \rightarrow 0$ and $\sigma \rightarrow 1$, which approaches the ordinary block code case, 
	the asymptotic sphere-packing bound becomes $\alpha(\A: 1,0,\delta)  \leq 1 - H_\A\left( \frac{\delta  }{2}\right)$ and the
	the upper bounds on 
	$\alpha(\A: \sigma,\lambda,\delta)$
	given in Theorems \ref{th:asympplotkin} and \ref{th:asymelias} become the asymptotic Plotkin and Elias bounds of \cite{GS04}, respectively.
\end{remark}
In Figures 1,2 and 3 we compare these asymptotic bounds for codes over $\Z_4$ and $\Z_8$.

\begin{figure}\label{fig:z4}
	\includegraphics[scale=0.65]{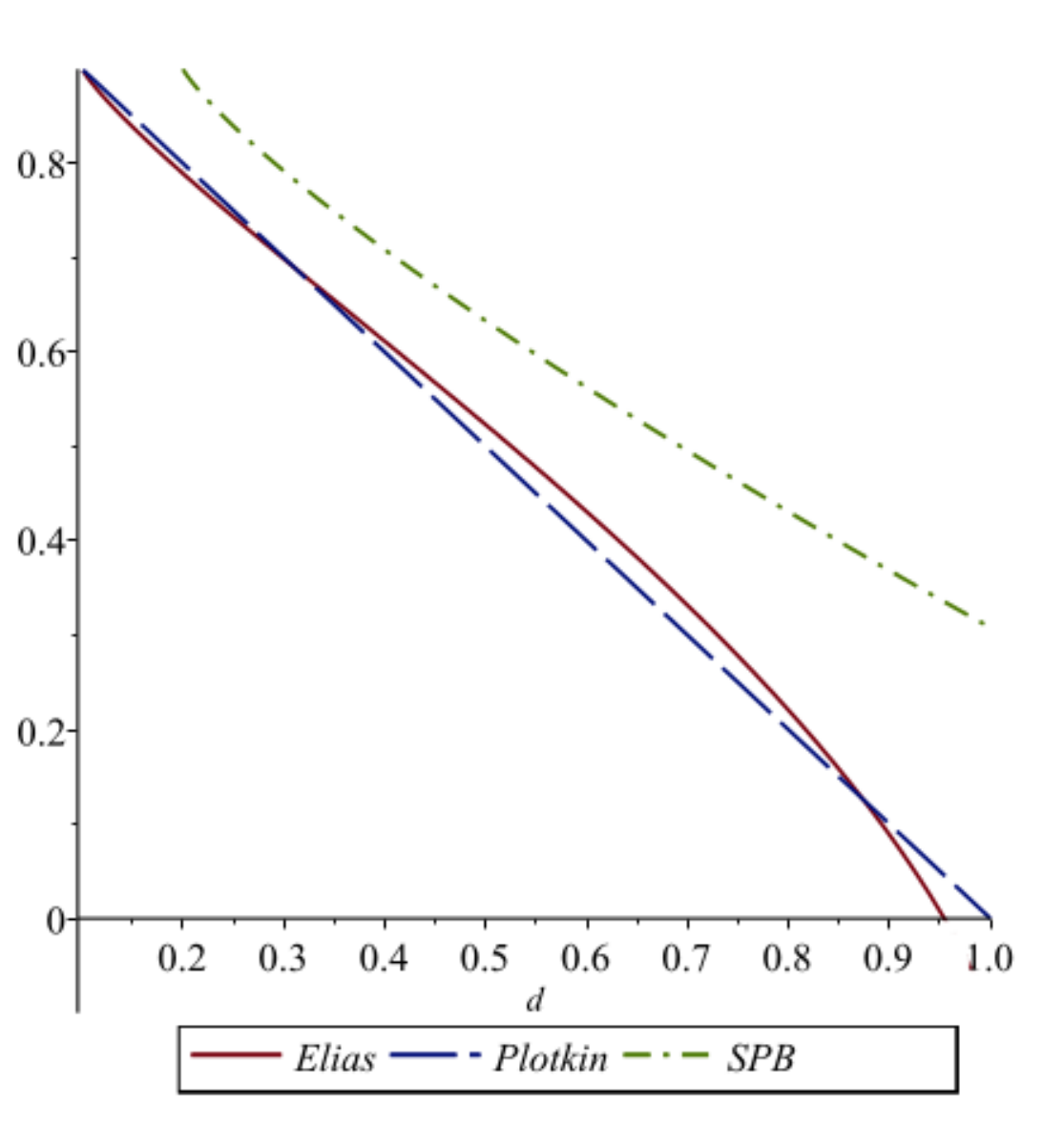}
	\caption{Asymptotic bounds for $R={\cal A}=\Z_4$, $\sigma=1,\lambda = 0.1$ for the induced Lee weight.}
\end{figure}

\begin{figure}\label{fig:z805}
	\includegraphics[scale=0.65]{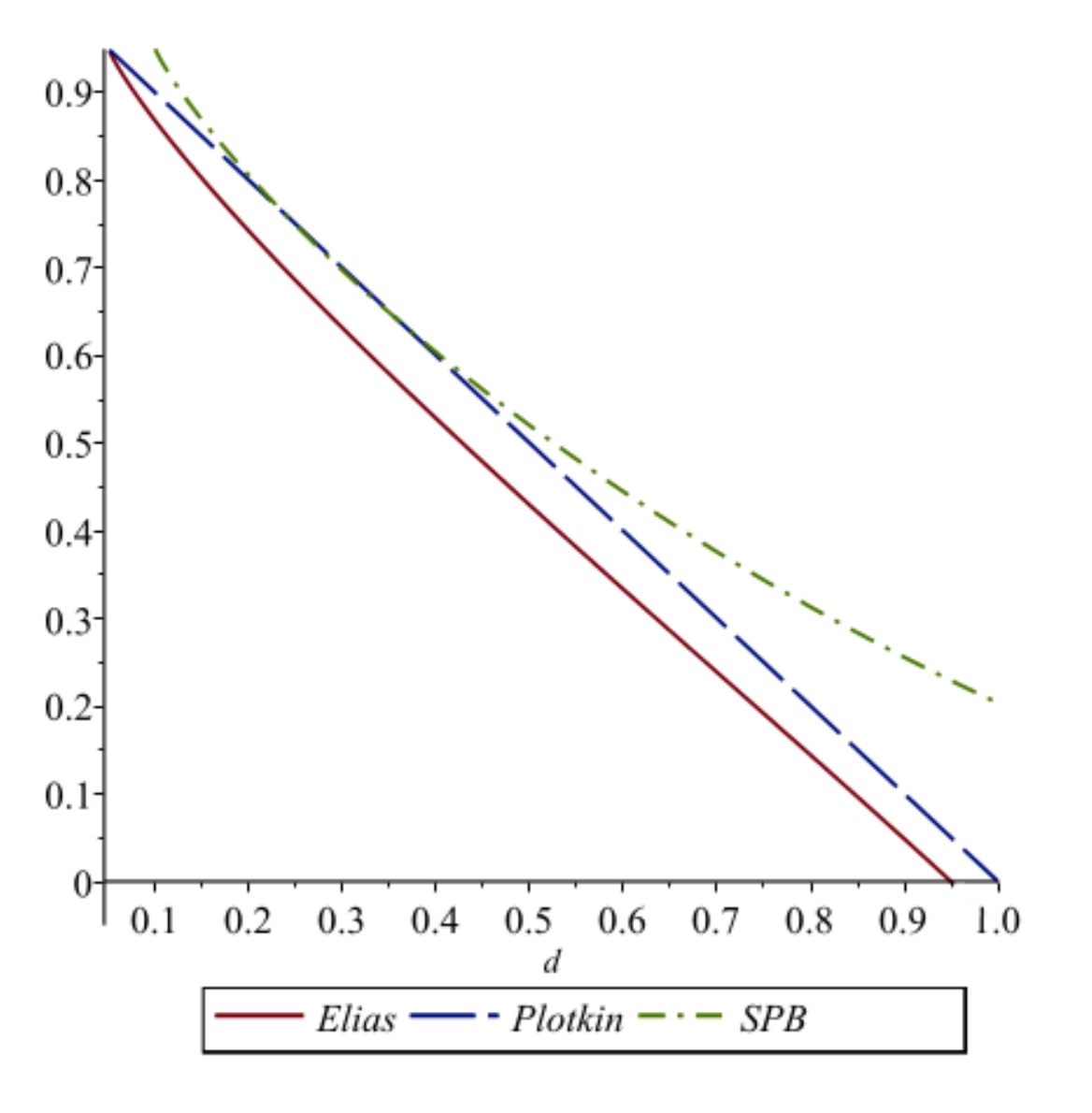}
	\caption{Asymptotic bounds for $R={\cal A}=\Z_8$, $\sigma=1,\lambda = 0.05$ for the induced homogeneous weight.}
\end{figure}

\begin{figure}\label{fig:z81}
	\includegraphics[scale=0.65]{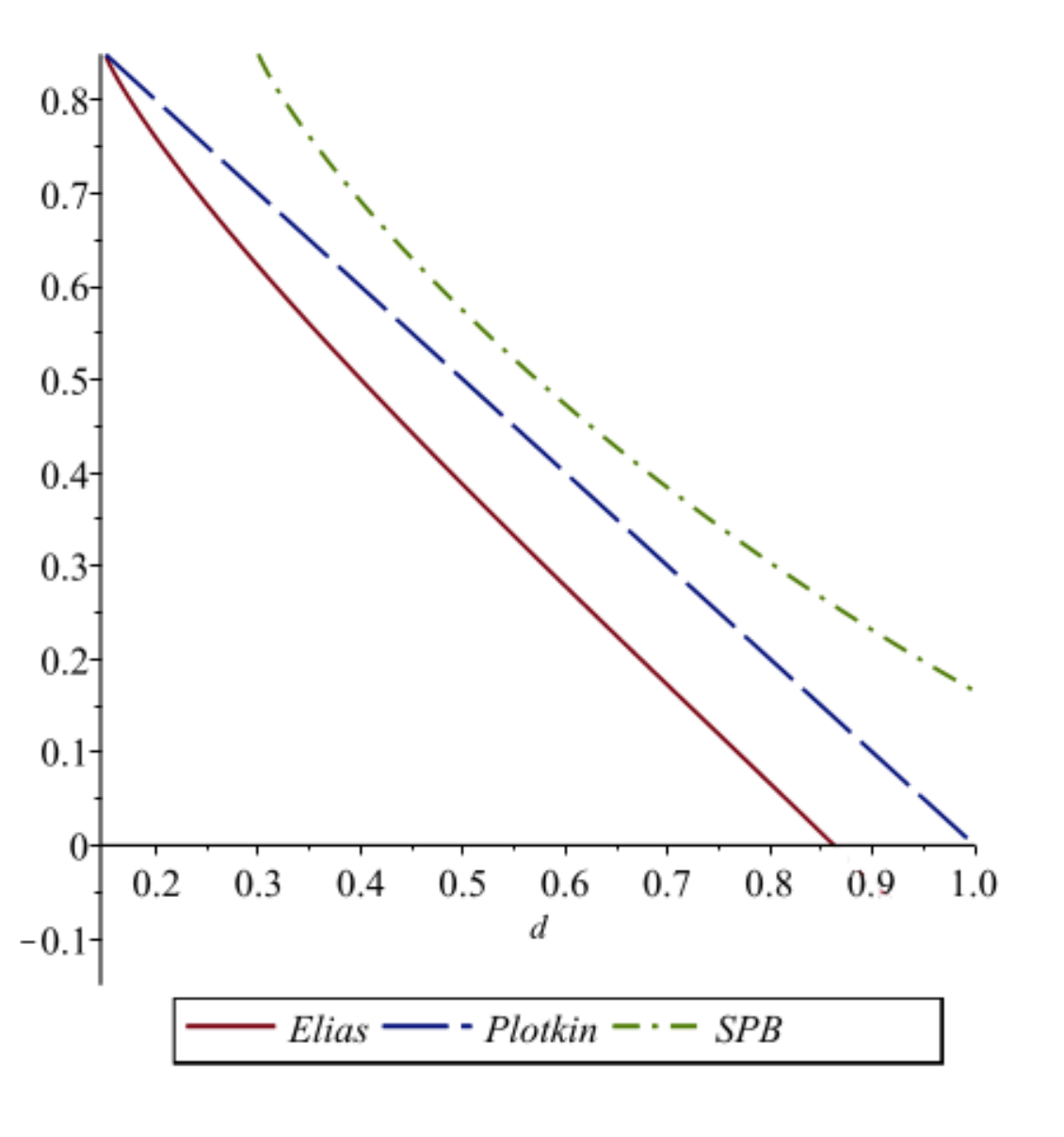}
	\caption{Asymptotic bounds for $R={\cal A}=\Z_8$, $\sigma=1,\lambda = 0.15$ for the induced homogeneous weight.}
\end{figure}

\section{Quotient Modules and Linear Network Coding}\label{sec:code} 

We define a communications network to be a directed acyclic graph ${\cal N} = (V,E)$, with a set of nodes or vertices $V$ and an edge set $E$ of order $n$. We assume that ${\cal N}$ has a single source node $s$ incident with some $m$ outgoing edges and has several sinks labelled by elements of a set ${\cal T}$. A message of $m$ packets is transmitted from the source node to be received at each sink in ${\cal T}$. For each sink $t \in {\cal T}$, let $E_t$ denote the set of edges incident with $t$ and let $|E_t|=n_t$. We furthermore assume that $n_t \geq m$ for each $t$ and indeed that there are at least $m$ edge disjoint paths connecting $s$ to $t$. It is assumed that the network topology is known by the sender and the sinks, since we are working with a coherent network model. 

\subsection{Linear Network Codes}

We now formally describe a network and what we mean by a linear code for a network, extending the definition given in \cite{YYN11} for the finite field case. Such a code will be a collection of codes (one for each sink). We do not require each code to itself be linear, the term `linear' in this context refers to the fact that data traverses the network via sequences of $R$-linear homomorphisms. In other words, at each node in the network, linear combinations of packets on its incoming edges are transmitted along its outgoing edges. If the set of messages is itself a vector space or module, as in the Example \ref{ex:toy}, then so will be each code, however, we do not impose this restriction in general.


The network ${\cal N}$ has $m$ edges with tail at the source node $s$ and $n$ edges in total. The set of messages of $\cal N$ is a subset ${\cal M}_0$ of ${\cal A}^m$. We do not assume that ${\cal M}_0$ is itself an $R$-module.
Each message
$x_0 \in {\cal M}_0$ corresponds to a unique network word $x=[x_0,0]$ in ${\cal M} :=\{ [u,0]: u \in {\cal M}_0\} \subset {\cal A}^n$, under the canonical embedding ${\cal A}^m$ into ${\cal A}^n$. The network itself may be identified with ${\cal A}^n$, where each $i$-th coordinate projection from ${\cal A}^n$ onto ${\cal A}$ corresponds to the $i$-th edge of the network, under some ordering.


A word $z \in {\cal A}^n$ is transmitted along the network 
by an $R$-automorphism of ${\cal A}^n$ that
can be realized as an invertible transfer map of the form:
$$\f :  {\cal A}^n \longrightarrow {\cal A}^n : z \mapsto \f(z)=(f_1(z),...,f_n(z)),$$
for some $R$-linearly independent epimorphisms $f_j \in \Hom({\cal A}^n,{\cal A}).$
If $x \in {\cal A}^n$ is transmitted from the source node $s$ and some edges of the network are corrupted by errors in the form of an error word 
$e \in {\cal A}^n$ then the network transmission is given by
$ y = \f(x + e).$ In other words, it is assumed that errors propagate through the network.

In \cite{KM03}, $\f$ is represented by an invertible transfer matrix $F \in R^{n \times n}, R=\F_q$, with respect to some fixed basis of $R^n$.
For the usual {\em scalar linear} network coding (as in \cite{YNY07,YY07,Z08}) we have the case $R ={\cal A} = \F_q$. 
For {\em vector linear} network coding often described in the literature (e.g. \cite{SKK10}), we usually see
$R = \F_q$ and ${\cal A} = \F_{q^r}$ or $\F_q^r$.


We now define a set of $R$-homomorphisms associated with each sink node. Each one will give rise to a code over ${\cal A}$.
For each sink $t \in {\cal T}$, let 
$$\Pi_t: {\cal A}^n \longrightarrow {\cal A}^{n_t} :z \mapsto (z_j)_{j \in E_t}$$ be the projection onto the coordinates indexed by the edges of the network incident with $t$. 
We further define maps $${\cal F}_t := \Pi_t \circ \f: {\cal A}^n \longrightarrow {\cal A}^{n_t}.$$
In the multicast setting, for each $t \in {\cal T}$, we require that ${\cal F}_t: {\cal M}\longrightarrow {\cal A}^{n_t}$ be an injection, in order that each sink $t$ can decode the transmitted word to the same unique message in ${\cal M}$. 
\begin{definition}
	Let 
	${\cal N}$ be a network with transfer map 
	${\cal F} \in {\rm Aut}_R({\cal A}^n)$. 
	The network code for node $t$ of $({\cal N},{\cal F})$ is the set
	$${\cal C}_t:=\{ \f_t(x) \in \A^{n_t}: x \in {\cal M}\} \subset {\cal A}^{n_t}.$$
	The network code of $({\cal N},{\cal F})$ is the collection
	${\cal C}:=\{{\cal C}_t : t \in {\cal T}\}$.
\end{definition}

We reiterate that neither ${\cal M}$ nor any ${\cal C}_t$ need be an $R$-module; the linearity of ${\cal C}$ refers only to $R$-linearity of the transfer map ${\cal F}$.

For the network message $x \in  {\cal A}^n$ and an error word $e \in {\cal A}^n$ the word received by node $t$ is
$y_t = \f_t(x+e). $
We denote by ${\cal K}_t$ the kernel of the map ${\cal F}_t$ in ${\cal A}^n$ so that $\cc_t = \f_t({\cal M}) \subset  {\cal A}^{n_t}$ which is isomomorphic as a bimodule to ${\cal A}^{n} / {\cal K}_t$. 
Observe that if $e \in {\cal K}_t$ then $y_t={\cal F}_t(x+e) = {\cal F}_t(x)$ is received as if no errors have occurred. If $m=n_t$, the kernel ${\cal K}_t$ is trivial and the decoder will not detect any errors.

Given the received word $y \in {\cal A}^{n_t}$, the decoder at node $t$ 
decides that $c = \f_t(x)$ has been transmitted
if $\dd_t(z,c)< \dd_t(z,c')$ for all $c' \in {\cal C}_t$.


\begin{example}
	Let $R = {\cal A} = \F_q$ and let ${\bf w}$ denote the usual Hamming weight on $\F_q^n$. 
	Let ${\cal N}$ be a network and let ${\cal C}_t$ be a network code for ${\cal N}$ at one of its sink nodes $t$. ${\cal F}_t$ has a representation as an $n \times n_t$ matrix $F_t$ with respect to a chosen basis. Then 
	${\bf w}_t(u) =\min \{ {\bf w}(z) : zF_t = u, \}={\bf w}(x+{\cal K}_t)$ counts the minimum number of linearly independent rows of $F_t$ required to obtain a representation of 
	$u = x F_t$. 
	If $x \in {\cal M}$ is transmitted and $y$ is received at
	$t$, the decoder will decode to $y-eF_t \in {\cal C}_t$ for some error $e \in \F_q^n$ of least Hamming weight satisfying $y=(x+e)F_t$.
	In other words, the decoder will assume that an adversary has tampered with 
	the least number of network edges resulting in a non-trivial contribution to the received word $y$.
\end{example}

Let $d_t$ denote the minimum distance of ${\cal C}_t$ with respect to $\dd_t$. 
Let $\ell_t$ be the size
of the support of ${\cal K}_t$.
For each $t \in {\cal T}$, we write ${\cal M}_t$ to denote the preimage of ${\cal C}_t$ in ${\cal A}^{n}$, so that ${\cal F}_t({\cal M}_t)={\cal C}_t$. We let $s_t$ denote the support of  ${\cal M}_t$.
We say that ${\cal C}$ is an $(n,\{(n_t,s_t \ell_t,|{\cal C}_t|, d_t): t \in {\cal T}\})$ network code.  
We define $$s({\cal C}):=\min \{ |{\cal C}_t| :t \in {\cal T}\},$$ which we call the size of $\cal C$, and seek upper bounds on this number, which is the effective maximum possible size of the message space ${\cal M}$.
For example, with respect to this notation, the network code described in Example \ref{ex:toy} is a 
$$(15,\{(2, 15,4,1), (3,15,4,1)\})$$ network code of size $4$.

\begin{definition}
	We denote by $A(n,\{( n_t,s_t,\ell_t,d_t ): t \in {\cal T}\})$ the maximum size $s({\cal C})$ of any $(n,\{(n_t, s_t,\ell_t,|{\cal C}_t|, d_t): t \in {\cal T}\})$ network code ${\cal C}$.
\end{definition}

There are network coding analogues of the sphere-packing, Gilbert-Varshamov and Singleton bounds over finite fields \cite{YNY07,YYN11}).
With the results presented here, we now add Plotkin and Elias-Bassalygo bounds.
For codes with alphabet $\F_2$, we give a graphical comparison of the Sphere-Packing, Singleton, Plotkin and Elias bounds, shown in Figure 4. 

\begin{figure}[h!]\label{fig:ffplot}	
	\includegraphics[scale=0.65]{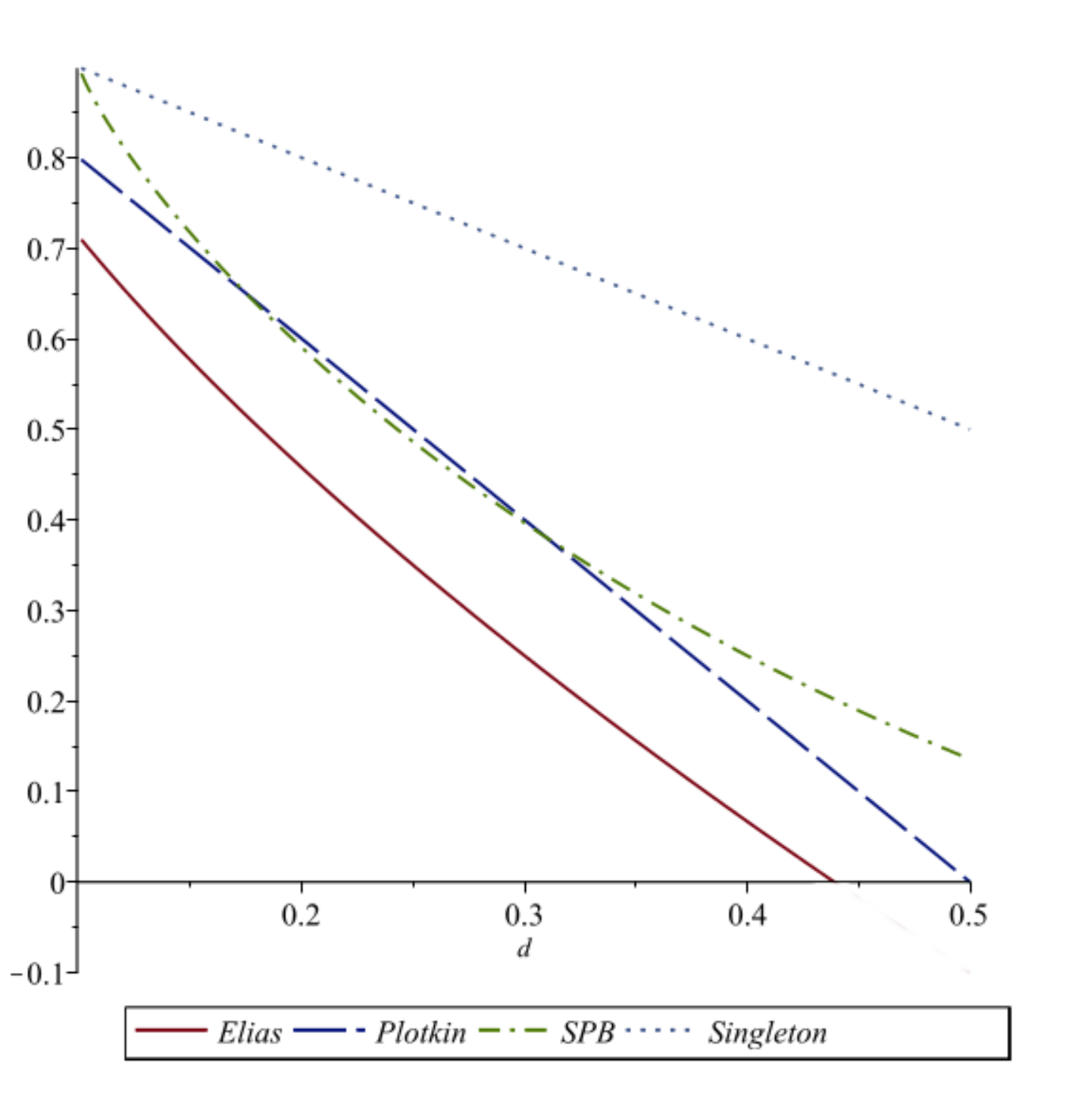}
	\caption{$R={\cal A} = \F_2$, $\sigma=1,\lambda = 0.1$}
\end{figure}


\providecommand{\bysame}{\leavevmode\hbox to3em{\hrulefill}\thinspace}

\end{document}